\documentclass[10pt]{article}

\usepackage{amssymb}
\usepackage{amsmath}
\usepackage{amsthm}
\usepackage{graphicx}
\usepackage{siunitx}
\usepackage{subfigure}

\newcommand{\R}{\mathbb R}

\newtheorem{theorem}{Theorem}[section]

\newtheorem{example}{Example}[section]
\newtheorem{definition}{Definition}[section]

\usepackage{xcolor,todonotes}

\author{M.~Schneider, J.~Lang}
\title{Well-Balanced and Asymptotic Preserving \\ IMEX-Peer Methods}
\author{
	Moritz Schneider\\
	{\small \it Technical University of Darmstadt, Dolivostraße 15, 64293 Darmstadt, Germany}\\
	{\small moschneider@mathematik.tu-darmstadt.de} \\ \\
	Jens Lang\footnote{corresponding author}\\
	{\small \it Technical University of Darmstadt, Dolivostraße 15, 64293 Darmstadt, Germany}\\
	{\small lang@mathematik.tu-darmstadt.de}}
\begin{document}
	\maketitle

	\begin{abstract}
		Peer methods are a comprehensive class of time integrators offering numerous degrees of freedom in their coefficient matrices that can be used to ensure advantageous properties, e.g. A-stability or super-convergence. In this paper, we show that implicit-explicit (IMEX) Peer methods are   well-balanced and asymptotic preserving by construction without additional constraints on the coefficients. These properties are relevant when solving  (the space discretisation of) hyperbolic systems of balance laws, for example. Numerical examples confirm the theoretical results and illustrate the potential of IMEX-Peer methods.
		
	\end{abstract}
	
	$\;$
	
	\noindent {\bf Keywords}: implicit-explicit (IMEX) Peer methods; well-balanced methods;
	asymptotic preserving methods
	
	\section{Introduction}
	
	Implicit-explicit (IMEX) Peer methods are designed to efficiently solve large systems of differential equations (ODEs)
	\begin{equation} \label{PeerForm}
	u' = F_0(u) + F_1(u), \qquad u(0) = u_0 \in \R^m, \quad m \geq 1
	\end{equation}
	that arise in the modelling of various dynamical processes in engineering, physics, chemistry and other areas. Due to their special structure, IMEX-Peer methods treat the non-stiff part $F_0$ explicitly and the stiff contribution $F_1$ implicitly, thus combining the advantage of lower costs for explicit schemes with the favourable stability of implicit solvers to enhance the overall computational efficiency.
	
	Peer methods are two-step methods with $s$ internal stages and belong to the class of general linear methods that were introduced and  described in detail by Butcher \cite{Butcher}. A specific feature of Peer methods is that all stages in each time step have the same order of consistency and, hence, order reduction is avoided.
	
	There is a wide range of literature concerning the different aspects of Peer methods and we will only give  a short overview. More details can be found in the introductory chapters of \cite{SchneiderLangHundsdorfer,SchneiderLangWeiner}. Peer methods were introduced  by Schmitt and Weiner in 2004 \cite{SchmittWeiner2004}. The construction of IMEX-Peer methods via extrapolation has been applied by several authors \cite{CardoneJackiewiczSanduZhang2014a,LangHundsdorfer2017}. An alternative construction using partitioned methods is given in \cite{SoleimaniKnothWeiner2017, ZhangSanduBlaise2014}. Since the coefficient matrices of Peer methods offer many degrees of freedom, the construction of  super-convergent schemes \cite{ SchneiderLangHundsdorfer, SoleimaniWeiner2018,  WeinerSchmittPodhaiskyJebens2009}  and the adaptation  to variable step sizes \cite{SchneiderLangWeiner,SoleimaniWeiner2017a} is possible.
	
	Throughout this paper, we consider $s$-stage Peer methods of the form
	\begin{equation} \label{PeerMethod}
	w_{n+1} = Pw_{n} + \Delta t \hat{Q}F_0(w_{n}) + \Delta t \hat{R}F_0(w_{n+1}) + \Delta t Q F_1(w_{n}) + \Delta t R F_1(w_{n+1})
	\end{equation}
	with $\hat{Q} = Q + RS_1$ and $\hat{R} = RS_2$ as given in \cite{SchneiderLangHundsdorfer}.
	Here, $P, Q,R,S_1$ and $S_2$ are  $s \times s$ coefficient matrices. The matrix $R$ is taken to be lower triangular with constant diagonal $\gamma > 0$ and $S_2$ is strictly lower triangular. The approximations in each time step are denoted by
	\begin{equation} \label{StageValues}
	w_n = \begin{bmatrix} w_{n,1}^T, \dots, w_{n,s}^T \end{bmatrix}^T \in \R^{ s \cdot m}, \qquad w_{n,i} \approx u(t_n + c_i \Delta t),
	\end{equation}
	where $t_n = n \Delta t, n \geq 0$ and the nodes $c_1, \dots, c_s \in \R$, corresponding to the $s$ stages, are such that $c_i \neq c_j$ if $i \neq j$ and $c_s = 1$. The application of $F_i$ is meant component-wise, i.e. $F_i(w_n) = \begin{bmatrix}
	F_i(w_{n,1})^T,   \dots,   F_i(w_{n,s})^T
	\end{bmatrix}^T, \; i=0,1$.  For the sake of notation, we use for an $s \times s$ matrix $M$ the same  symbol for its Kronecker product with the $m \times m$ identity matrix $M \otimes I_m$ as a mapping from the space $\R^{s \cdot m}$ to itself.
	An extensive analysis of consistency and stability along with the construction of super-convergent methods as well as the adaption to variable step sizes is given in \cite{SchneiderLangHundsdorfer,SchneiderLangWeiner}.
	
	In this paper, we show that our recently developed super-convergent IMEX-Peer methods \cite{ SchneiderLangHundsdorfer} possess two additional properties that are important when dealing with hyperbolic balance and conservation laws: They are \emph{well-balanced} and \emph{asymptotic preserving}. We restrict the analysis to the setting of constant step sizes. However, the results   hold true for Peer methods applied with variable step sizes as well.
	
	For further investigation, we follow the approach of Boscarino and Pareschi \cite{BoscarinoPareschi} and consider the hyperbolic system of balance laws
	\begin{equation}\label{BalanceLaw}
	U_t + F(U)_x = G(U),
	\end{equation}
	where $U\in\R^N$ and $F, G : \R^N \to \R^N$. Usually, $F(U)$ gives the flux and $G(U)$ is the source.
	
	The remainder is organised as follows. In Section \ref{WellBalanced}, we show that IMEX-Peer methods are well-balanced   without  additional constraints on the coefficients. The same holds true for the asymptotic preservation property as analysed in Section \ref{AsymptoticPerserving}. Numerical experiments to illustrate the theoretical results are given in Section \ref{NumericalResults}.
	
	\break

	\section{Well-Balanced IMEX-Peer Methods} \label{WellBalanced}

	The steady-state $U^*$ of the hyperbolic system of balance laws \eqref{BalanceLaw} is characterized by
	\begin{equation} \label{SteadyState}
	U^*_t \equiv 0 \iff F(U^*)_x = G(U^*).
	\end{equation}
	Accordingly, a numerical scheme is called \emph{well-balanced}, if it preserves the steady-state solution $U^*$ as characterized in \eqref{SteadyState}. Since Peer methods are time integrators, we focus on the influence of  time discretisation on the well-balanced property of the numerical solution. Thus, we discretise \eqref{BalanceLaw} in space and obtain the  system of ODEs
	\begin{equation} \label{ODE}
	u'(t) = F_0(u(t)) + F_1(u(t))
	\end{equation}
	with non-stiff (flux) function $F_0$ and stiff source term $F_1$. Analogously to \eqref{SteadyState}, the steady-state $u^*$ is now described by
	\begin{equation*}
	{u^*}'(t) \equiv 0 \iff F_0(u^*) + F_1(u^*) = 0.
	\end{equation*}
	Assume that the numerical solution of  \eqref{ODE} yields an   approximation $v_n \approx u(t_n) \in \R^m$ satisfying
	\begin{equation*}
	F_0(v_n) + F_1(v_n) = 0.
	\end{equation*}
	Then, in some sense  $ v_n' = 0$ holds and, in order to  capture the steady-state,  we claim
	\begin{equation*}
	u(t_{n+1})\approx v_{n+1} = v_n.
	\end{equation*}
	This concept works well for one-step methods as discussed in \cite{BoscarinoPareschi}. Since we are dealing with two-step methods, a small modification is needed to take into account all values of the previous time step. This leads to the following definition of well-balanced IMEX-Peer methods.
	\begin{definition}\label{DefWellBalanced}
		An $s$-stage IMEX-Peer method \eqref{PeerMethod} is called \emph{well-balanced} if
		\begin{equation} \label{CheckWellBalanced}
		F_0(w_n) + F_1(w_n) =0
		\end{equation}
		implies $w_{n+1} = w_n$, where $w_{n,1} = \dots = w_{n,s}$.
	\end{definition}
	Now, we can prove the following theorem.
	\begin{theorem}\label{ThmWellBalanced}
		IMEX-Peer methods of the form \eqref{PeerMethod} with coefficient matrices that satisfy the standard consistency conditions from \cite{SchneiderLangHundsdorfer}
		\begin{equation*}
		Pe = e \qquad \textrm{and} \qquad S_1 = (I_s-S_2)V_0V_1^{-1},
		\end{equation*}
		where $V_0  = (c_i^{j-1})_{i,j}$ and $V_1   = ((c_i-1)^{j-1})_{i,j}$,
		are \emph{well-balanced}
		in the sense of Definition \ref{DefWellBalanced},
		given that \eqref{PeerMethod} has a unique solution $w_{n+1}$ for  $\Delta t$ sufficiently small.
	\end{theorem}
	\begin{proof}
		By Definition \ref{DefWellBalanced}, we have $F_0(w_n) + F_1(w_n) = 0$ with $w_n = e \otimes w_{n,s}$ and $F_i(w_n) = e \otimes F_i(w_{n,s}), \; i = 0,1,$ 	where $e = (1, \dots, 1)^T \in \R^s.$
		Under the hypotheses of Theorem \ref{ThmWellBalanced} stated above, we prove that  \eqref{PeerMethod} implies $w_{n+1} = w_n$:
		\begin{align*}
		w_{n+1} & =Pw_n + \Delta t (\hat{Q}F_0(w_n) + Q F_1(w_n)) + \Delta t (\hat{R}F_0(w_{n+1}) + R F_1(w_{n+1}) )  \\
		& = (P \otimes I_m)(e \otimes w_{n,s}) + \Delta t  ( (R(S_1 + S_2 - I_s)) \otimes I_m)   (e \otimes F_0(w_{n,s}) ).
		\end{align*}
		Using $R(S_1 + S_2 - I_s) = R(I_s-S_2)(V_0V_1^{-1}-I_s)$ and $V_0V_1^{-1}e = V_0e_1 = e,$ where $e_1 = (1,0, \dots, 0)^T \in \R^s$, the second term vanishes and we obtain $w_{n+1} = w_n$.
	\end{proof}
	
	In practice, we cannot expect  \eqref{CheckWellBalanced} to hold for all stage values of $w_n$ but rather that when the numerical solution converges to the steady-state, we will reach a point in time when all stage values are sufficiently similar and the last stage of the time step satisfies
	$  F_0(w_{n,s}) + F_1(w_{n,s}) = 0$.
	Then, the numerical scheme should reproduce the steady-state. It can be shown that if
	\begin{equation*}
	F_0(w_{n,s}) + F_1(w_{n,s}) = 0 \quad \textrm{ and } \quad w_{n,i} = w_{n,s} + \mathcal{O}(\varepsilon), \quad i = 1, \dots, s-1,
	\end{equation*}
	we obtain for continuous $F_0$ and $F_1$, analogously to the proof of Theorem \ref{ThmWellBalanced},
	\begin{equation*}
	w_{n+1} = w_n + \mathcal{O}(\varepsilon).
	\end{equation*}
	Hence, the well-balanced property is beneficial in practical applications even if the strong condition \eqref{CheckWellBalanced} is not fulfilled exactly.

	\section{Asymptotic Preserving IMEX-Peer Methods} \label{AsymptoticPerserving}
	
	We investigate the behaviour of IMEX-Peer methods when the hyperbolic balance laws \eqref{BalanceLaw} are scaled with a parameter $\varepsilon >0$. This is discussed in detail by Chen, Levermore and Liu \cite{ChenLevermoreLiu1994} and has been adopted for  IMEX Runge-Kutta methods by Boscarino and Pareschi   and for  multistep methods by Dimarco and Pareschi \cite{BoscarinoPareschi, DimarcoPareschi}.
	
	Scaling the space and time variables in \eqref{BalanceLaw} with a parameter $\varepsilon >0$ leads to
	\begin{equation} \label{ScaledConsLaws}
	U^\varepsilon_t + F(U^\varepsilon)_x = \frac{1}{\varepsilon}G(U^\varepsilon).
	\end{equation}
	We are interested in the performance of numerical schemes that solve \eqref{ScaledConsLaws} for $\varepsilon \to 0$.
	Taking the limit $\varepsilon \to 0$ analytically yields the system of algebraic equations
	\begin{equation} \label{SourceEq}
	G(U^0) = 0.
	\end{equation}
	Following  the analysis in \cite{BoscarinoPareschi, ChenLevermoreLiu1994}, we assume that $G(U)$ with $U \in \R^N$ is a dissipative relaxation operator, i.e., there exists an $M \times N$ matrix $C$ with $\operatorname{rank}(C) = M < N$ and
	\begin{equation} \label{Qtilde}
	CG(U) = 0 \qquad \textrm{ for all } U \in \R^N.
	\end{equation}
	We set $u = CU \in \R^M$ to be the  vector of conservation quantities.
	Further, each such $u$ uniquely defines a local equilibrium value
	\begin{equation} \label{locEqVal}
	U = E(u)
	\end{equation}
	that satisfies
	\begin{equation}\label{CondGEu}
	0= G(E(u)) = G(U)
	\end{equation} and $u = CE(u) = CU.$
	
	Therefore, for every solution $U^0$ of \eqref{SourceEq}, we find a uniquely determined vector of conserved quantities $u^0$ such that $U^0 = E(u^0).$
	Going back to \eqref{ScaledConsLaws} and multiplying with $C$, we obtain
	\begin{equation*}
	CU^\varepsilon_t + CF(U^\varepsilon)_x = \frac{1}{\varepsilon}CG(U^\varepsilon)
	\end{equation*}
	and, hence, for $\varepsilon \to 0$,   a system of $M$ conservation laws
	\begin{equation*}
	(CU^0)_t + (CF(U^0))_x = 0.
	\end{equation*}
	Using the equilibrium approximation $U^0 = E(u^0)$ with $CE(u^0) = CU^0 = u^0$ from above, we have
	\begin{equation*}
	(CE(u^0))_t + (CF(E(u^0)))_x = 0
	\end{equation*}
	and, finally, obtain the typical system of conservation laws
	\begin{equation} \label{EquilibSyst}
	u^0_t + f(u^0)_x = 0
	\end{equation}
	with $f(\cdot)  = CF(E(\cdot)).$
	
	In the following, we verify that IMEX-Peer methods, as defined in \eqref{PeerMethod} and \eqref{StageValues},  capture the asymptotic behaviour described  above. To this end, we write  \eqref{ScaledConsLaws} in the standard form for IMEX-Peer methods \eqref{PeerForm} and define
	\begin{equation}\label{Ident}
	F_0(U^\varepsilon)    = -F(U^\varepsilon)_x  \qquad \textrm{and} \qquad
	F_1(U^\varepsilon)  = \frac{1}{\varepsilon} G(U^\varepsilon),
	\end{equation}
	where we identify  $U^\varepsilon, F(U^\varepsilon)_x$ and $G(U^\varepsilon)$ with the corresponding spatial discretisations for the sake of enhanced readability.
	Applying an IMEX-Peer method \eqref{PeerMethod} to \eqref{ScaledConsLaws} and \eqref{Ident} gives
	\begin{equation}\label{IMEXapplied}
	U^\varepsilon_{n+1} = PU^\varepsilon_{n} - \Delta t \hat{Q}F(U^\varepsilon_{n})_x - \Delta t \hat{R}F(U^\varepsilon_{n+1})_x + \frac{\Delta t}{\varepsilon} Q G(U^\varepsilon_{n}) + \frac{\Delta t}{\varepsilon} R  G(U^\varepsilon_{n+1}).
	\end{equation}
	As in the continuous case,   $\varepsilon \to 0$ yields
	\begin{equation} \label{IMEXlim0}
	QG(U^0_{n}) + RG(U^0_{n+1}) = 0.
	\end{equation}
	At this point, we are faced with a typical problem occurring for multi-step methods: We have to introduce an additional condition on the values of the previous time step $U^\varepsilon_{n}$. This is reasonable since the contribution of $U^\varepsilon_{n}$ via $QG(U^\varepsilon_{n})$ depends on the specific choice of $G(\cdot)$ and cannot be compensated by $RG(U^\varepsilon_{n+1})$ independently of $G(\cdot)$. Hence, we claim \emph{well-prepared} initial values \cite{BoscarinoPareschi,FilbertJin}.
	
	\begin{definition} \label{WellPrepared}
		The initial data $U^\varepsilon_{n}$ for \eqref{IMEXapplied} is said to be \emph{well-prepared} if
		\begin{equation*}
		U^\varepsilon_{n} = E(u^\varepsilon_{n}) + \mathcal{O}(\varepsilon).
		\end{equation*}
	\end{definition}
	This allows us to formulate the following.
	
	\begin{theorem}
		Assume the initial data is well-prepared. Then, in the limit $\varepsilon \to 0$, an IMEX-Peer method \eqref{PeerMethod} applied to \eqref{ScaledConsLaws} becomes the explicit Peer method $(F_1 \equiv 0)$ applied to the equilibrium system \eqref{EquilibSyst}.	
	\end{theorem}
	
	\begin{proof}
		
		Since the initial values  are well-prepared, we obtain for the limit $\varepsilon \to 0$
		\begin{equation*}
		RG(U^0_{n+1}) = 0 \implies G(U^0_{n+1}) = 0
		\end{equation*}
		from \eqref{CondGEu} and  \eqref{IMEXlim0} since $R$ is regular.
		Analogously to the  continuous case, we define
		\begin{equation*}
		u^0_{n+1} = (I_s \otimes C)U^0_{n+1} = \begin{bmatrix}
		(CU^0_{n+1,1})^T, \dots, (CU^0_{n+1,s})^T
		\end{bmatrix}^T.
		\end{equation*}
		We set $C \bullet U = (I_s \otimes C)U$ for any $U\in \R^{N\cdot s}$.
		As in the continuous case \eqref{locEqVal}, the local equilibrium values $U^0_{n+1}$ are defined by $u^0_{n+1}$ via
		\begin{equation*}
		U^0_{n+1} = E(u^0_{n+1})
		\end{equation*}
		where $G(E(u^0_{n+1})) = 0$ and $C \bullet E(u^0_{n+1}) = u^0_{n+1}$.
		
		Multiplying \eqref{IMEXapplied} with $C$ and using that for any  $M \in \R^{s\times s}$
		\begin{align*}
		C \bullet (M \otimes I_N) \begin{bmatrix}
		U_{n,1}^{\varepsilon T},  \dots, U_{n,s}^{\varepsilon T}
		\end{bmatrix}^T = (M \otimes I_M)\left(C \bullet \begin{bmatrix}
		U_{n,1}^{\varepsilon T},  \dots,  U_{n,s}^{\varepsilon T}
		\end{bmatrix}^T\right),
		\end{align*}
		as well as the notation $MU = (M \otimes I_N)U$ for any $U \in \R^{N\cdot s}$, we find
		\begin{align} \label{QAplPeer}
		C\bullet U^\varepsilon_{n+1} & =   P (C \bullet  U^\varepsilon_{n})  - \Delta t \hat{Q}\left(C \bullet F(U^\varepsilon_{n})_x\right)- \Delta t  \hat{R} \left(C \bullet F(U^\varepsilon_{n+1})_x\right) \nonumber \\
		& +  \frac{\Delta t}{\varepsilon}  Q \left(  C\bullet  G(U^\varepsilon_{n})\right) +  \frac{\Delta t}{\varepsilon}   R \left( C \bullet G(U^\varepsilon_{n+1})\right).
		\end{align}
		We recall from \eqref{Qtilde} that $C \bullet G(U)  = 0$ for all $U \in \R^{N\cdot s}$. Hence, \eqref{QAplPeer} reduces to
		\begin{equation*}
		C\bullet U^\varepsilon_{n+1}   =   P (C \bullet U^\varepsilon_{n})  - \Delta t \hat{Q}\left(C \bullet F(U^\varepsilon_{n})_x\right) - \Delta t  \hat{R}\left( C \bullet F(U^\varepsilon_{n+1})_x\right).
		\end{equation*}
		For $\varepsilon \to 0$, we replace $U^\varepsilon_n$ by $E(u^0_n)$ and use $C \bullet E(u^0_n) = u^0_n$ to obtain
		\begin{equation} \label{IMEXlim}
		u^0_{n+1}   =   P u^0_{n}  - \Delta t \hat{Q}(C\bullet F(E(u^0_{n}))_x) - \Delta t  \hat{R}(C \bullet F(E(u^0_{n+1}))_x).
		\end{equation}
		Setting $f(\cdot) = C \bullet F(E(\cdot)),
		$
		we observe that \eqref{IMEXlim} coincides with the application of IMEX-Peer method \eqref{PeerMethod} to the system of conservation laws
		\eqref{EquilibSyst}
		\begin{equation} \label{Expl}
		u^0_t = -f(u^0)_x,
		\end{equation}
		where   $F_0(u^0) = - f(u^0)_x$ and $F_1 \equiv 0$, giving an explicit scheme.
	\end{proof}
	We see that instead of solving \eqref{Expl} explicitly, we can equivalently apply an IMEX-Peer method to  the relaxed system \eqref{ScaledConsLaws}, thus profiting from the stabilisation of the implicit part.
	
	\section{Numerical Examples} \label{NumericalResults}
	
	We present two numerical examples to illustrate that IMEX-Peer methods are well-balanced and asymptotic preserving. Since we focus on time integration, we consider systems of ODEs where no spatial discretisation is necessary and apply our recently developed super-convergent methods from \cite{SchneiderLangHundsdorfer}.
	
	\begin{example}[Well-balanced Property]
		We demonstrate the effect of the well-balanced property using a  system of ODEs of form \eqref{PeerForm} with non-stiff part $F_0(u) = [u_2, -u_1]^T$ and stiff contribution $ F_1(u) =   [0,  1 - u_2]^T$ as introduced by Boscarino and Pareschi in \cite{BoscarinoPareschi}. The unique equilibrium point is $u^* = [1,0]^T$.
		
		In Figure \ref{WellFig}, the behaviour  of the numerical approximation of the solution components $u_1$ and $u_2$ for $t \in [0,15]$ is given using IMEX-Peer methods from \cite{SchneiderLangHundsdorfer}. Starting with $u_0 = [0,1]^T$, we observe that the IMEX-Peer methods reach the equilibrium after a short time even for a large step size $\Delta t = 1$. Boscarino and Pareschi demonstrate in \cite{BoscarinoPareschi} that this is usually not the case if the time integrator is not well-balanced.
	\end{example}
	
	\begin{example}[Asymptotic Preservation Property] To verify that our super-convergent IMEX-Peer methods developed in \cite{SchneiderLangHundsdorfer} are asymptotic preserving, we consider a stiff system of ODEs  \eqref{PeerForm}
		where $F_0(u) = [-u_2, u_1]^T$ and $\displaystyle F_1(u) =  \frac1\varepsilon [0,  \sin u_1 - u_2]^T$ with scaling parameter $\varepsilon > 0$, initial values $u(0) = [\pi/2,1]^T$ and $t \in [0,5]$ as given by Pareschi and Russo in \cite{PareschiRusso}.
		
		Numerical results for $\Delta t = 0.2 \cdot 2^{-i}, i = 0, \dots, 4$ and  $\varepsilon = 1, 10^{-5}$ are given in Figure \ref{AsymFig}. The error is computed using the scaled maximum error norm over all time steps $err = \max_{0\leq t_n\leq 5} \max_{i=1,2} |U_i - u_i|/(1 + |u_i|)$, where $U$ is the approximate solution and $u$ is a reference solution computed using the \textsc{Matlab} routine \textsc{ODE15s}.
		
		We observe that the orders of convergence  of the super-convergent methods IMEX-Peer2s, IMEX-Peer3s and IMEX-Peer4s with stage number $s = 2,3,4$ from \cite{SchneiderLangHundsdorfer}, derived using a least squares fit, are 2.9, 3.9, 5.2 for $\varepsilon =1$ and 3.0, 4.0, 4.8 for $\varepsilon = 10^{-5}$. Hence, they match   the theoretical orders $s+1$ nicely.
		
		We remark that the order of convergence is affected in an intermediate region $\Delta t = \mathcal{O}(\varepsilon)$, see \cite{DimarcoPareschi,PareschiRusso} for further details. Nevertheless, the order of convergence  is fully  restored when $\Delta t$ leaves the regime $\mathcal{O}(\varepsilon)$, therefore, the drawback is negligible \cite{PareschiRusso}.
	\end{example}
	
	\begin{figure}
		\centering
		\subfigure[Results for the well-balanced test. All meth- $\qquad$ ods   capture the equilibrium after short   time  and $\qquad$  $\qquad$  for a large time step $\Delta t = 1$.]{\label{WellFig}	\includegraphics[scale=.52]{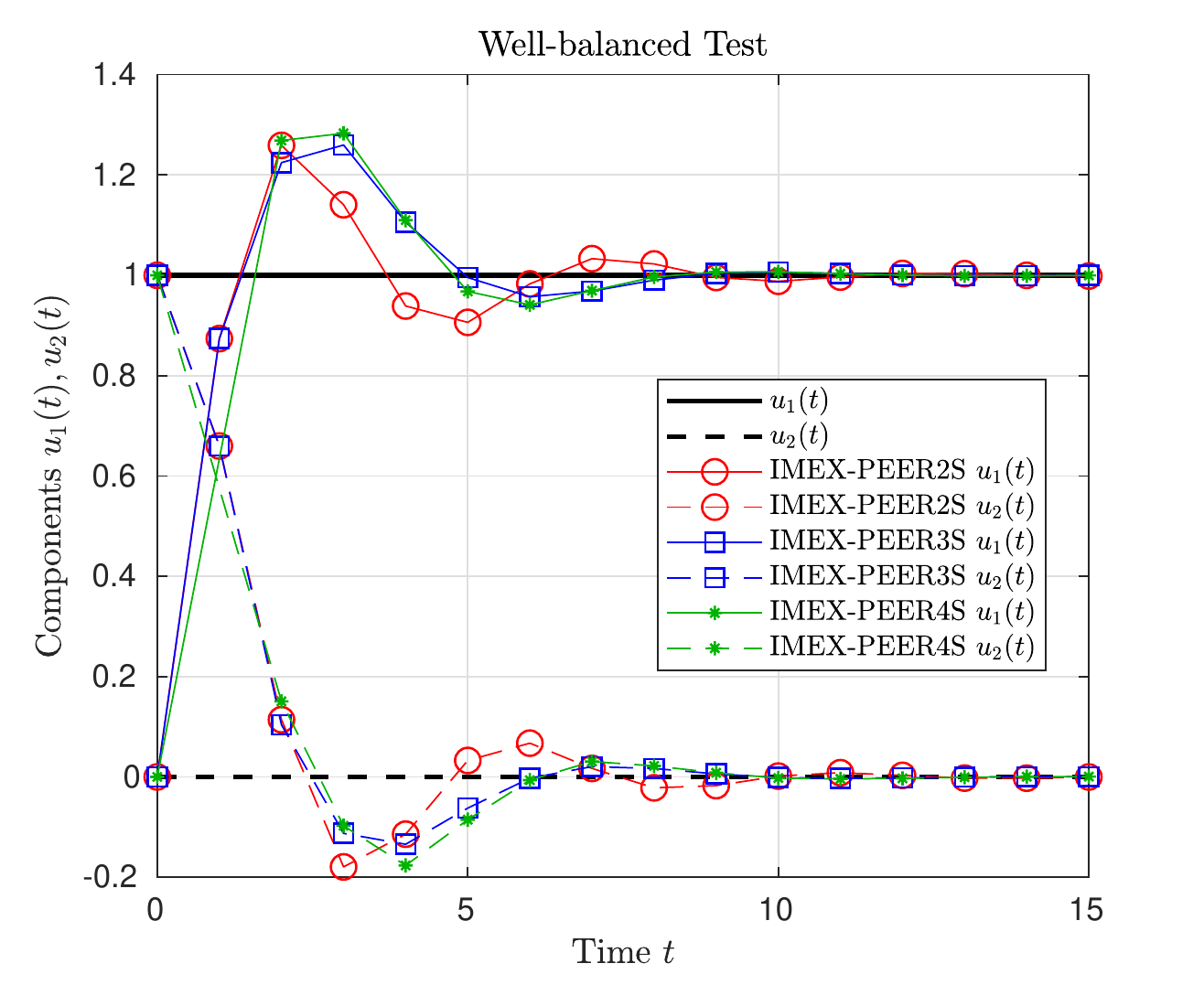} }
		\subfigure[Results for the asymptotic preservation test. The super-convergent  methods keep their  order of $s+1$ for various choices of $\varepsilon$.]{\label{AsymFig} \includegraphics[scale=.5]{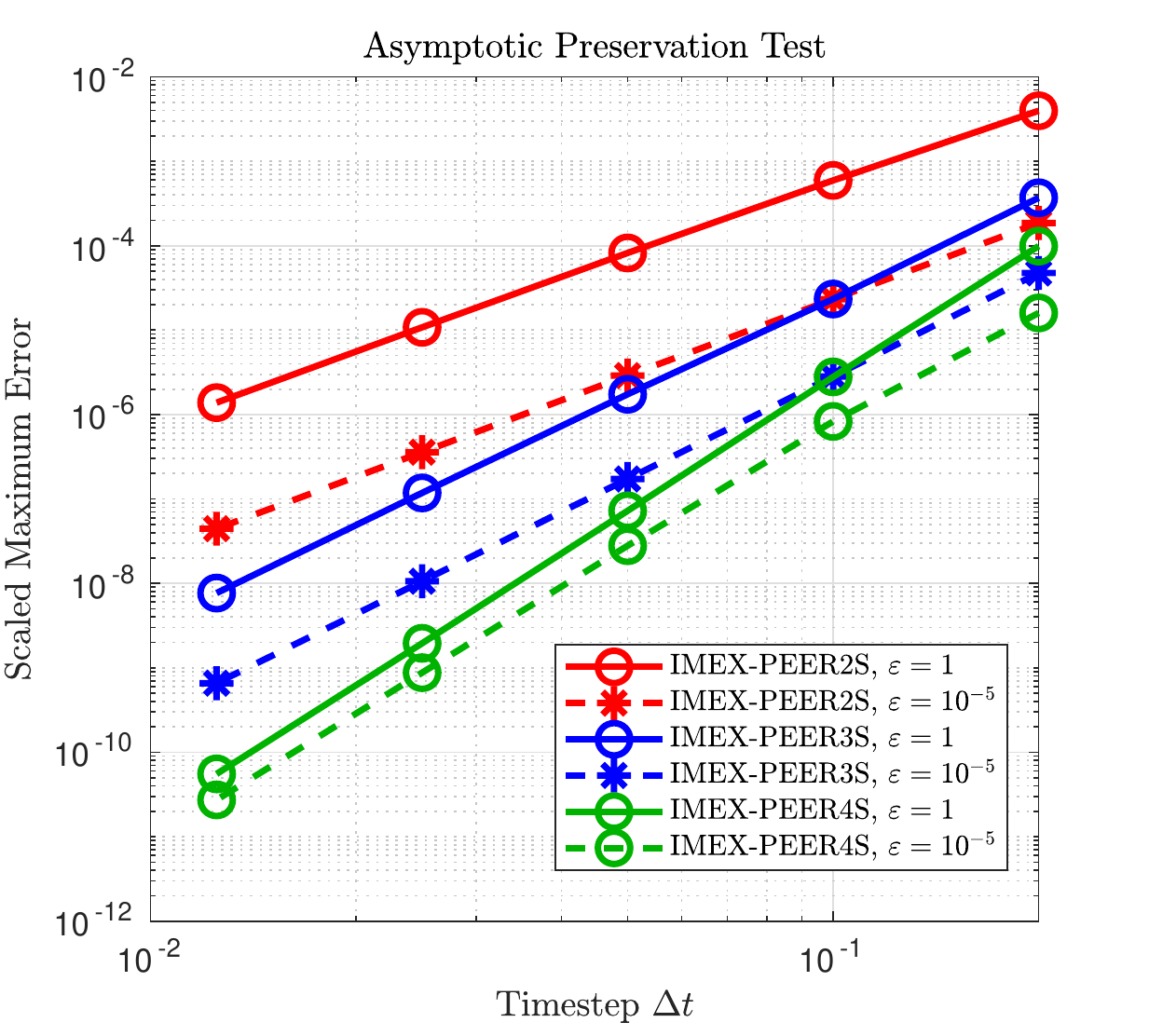} }
		\caption{Numerical results for  super-convergent and A-stable $s$-stage IMEX-Peer methods IMEX-Peer2s, IMEX-Peer3s, and IMEX-Peer4s from \cite{SchneiderLangHundsdorfer}. }
	\end{figure}
	
	In conclusion, we have shown that IMEX-Peer methods are well-balanced and asymptotic preserving by construction and, hence, suitable for the preservation of steady-states and the capture of asymptotic limits for space-time scaling.

	\section{Acknowledgement}
		J. Lang was supported by the
		German Research Foundation within the collaborative research center
		TRR154 ``Mathematical Modeling, Simulation and Optimisation Using
		the Example of Gas Networks'' (DFG-SFB TRR154/2-2018, TP B01) and
		the Graduate Schools Computational Engineering (DFG GSC233)
		and Energy Science and Engineering (DFG GSC1070).

		\bibliographystyle{plain}
		\bibliography{bibimexpeerwellbalanced}

\begin{thebibliography}{10}

\bibitem{BoscarinoPareschi}
S.~Boscarino and L.~Pareschi.
\newblock On the asymptotic properties of {IMEX} {R}unge--{K}utta schemes for
  hyperbolic balance laws.
\newblock {\em J. Comput. Appl. Math.}, 316:60--73, 2017.

\bibitem{Butcher}
J.C. Butcher.
\newblock General linear methods.
\newblock {\em Acta Numerica}, 15:157--256, 2006.

\bibitem{CardoneJackiewiczSanduZhang2014a}
A.~Cardone, Z.~Jackiewicz, A.~Sandu, and H.~Zhang.
\newblock Extrapolation-based implicit-explicit general linear methods.
\newblock {\em Numer. Algorithms}, 65:377--399, 2014.

\bibitem{ChenLevermoreLiu1994}
G.~Chen, C.~Levermore, and T.~Liu.
\newblock Hyperbolic {C}onservation {L}aws with {S}tiff {R}elaxation {T}erms
  and {E}ntropy.
\newblock {\em Commun. Pure Appl. Math.}, 47:787--830, 1994.

\bibitem{DimarcoPareschi}
G.~Dimarco and L.~Pareschi.
\newblock Implicit-explicit linear multistep methods for stiff kinetic
  equations.
\newblock {\em SIAM J. Num. Anal.}, 55:664--690, 2017.

\bibitem{FilbertJin}
F.~Filbert and S.~Jin.
\newblock A class of asymptotic-preserving schemes for kinetic equations and
  related problems with stiff sources.
\newblock {\em J. Comp. Phys.}, 229:7625--7648, 2010.

\bibitem{LangHundsdorfer2017}
J.~Lang and W.~Hundsdorfer.
\newblock Extrapolation-based implicit-explicit {P}eer methods with optimised
  stability regions.
\newblock {\em J. Comp. Phys.}, 337:203--215, 2017.

\bibitem{SchmittWeiner2004}
B.A. Schmitt and R.~Weiner.
\newblock Parallel two-step {W}-methods with peer variables.
\newblock {\em SIAM J. Numer. Anal.}, 42(1):265--282, 2004.

\bibitem{SchneiderLangHundsdorfer}
M.~Schneider, J.~Lang, and W.~Hundsdorfer.
\newblock Extrapolation-based super-convergent implicit-explicit {P}eer methods
  with {A}-stable implicit part.
\newblock {\em J. Comp. Phys.}, 367:121--133, 2018.

\bibitem{SchneiderLangWeiner}
M.~Schneider, J.~Lang, and R.~Weiner.
\newblock Super-convergent implicit–explicit {P}eer methods with variable
  step sizes.
\newblock {\em J. Comput. Appl. Math.}, doi: 10.1016/j.cam.2019.112501, 2019.

\bibitem{SoleimaniKnothWeiner2017}
B.~Soleimani, O.~Knoth, and R.~Weiner.
\newblock {IMEX} {P}eer methods for fast-wave-slow-wave problems.
\newblock {\em Appl. Numer. Math.}, 118:221--237, 2017.

\bibitem{SoleimaniWeiner2017a}
B.~Soleimani and R.~Weiner.
\newblock A class of implicit {P}eer methods for stiff systems.
\newblock {\em J. Comput. Appl. Math.}, 316:358--368, 2017.

\bibitem{SoleimaniWeiner2018}
B.~Soleimani and R.~Weiner.
\newblock Superconvergent {IMEX} {P}eer methods.
\newblock {\em Appl. Numer. Math.}, 130:70--85, 2018.

\bibitem{PareschiRusso}
D.~Trigiante.
\newblock {\em Recent trends in numerical analysis}.
\newblock Nova Science Publishers, New York, 2000.

\bibitem{WeinerSchmittPodhaiskyJebens2009}
R.~Weiner, B.A. Schmitt, H.~Podhaisky, and S.~Jebens.
\newblock Superconvergent explicit two-step {p}eer methods.
\newblock {\em J. Comput. Appl. Math.}, 223:753--764, 2009.

\bibitem{ZhangSanduBlaise2014}
H.~Zhang, A.~Sandu, and S.~Blaise.
\newblock Partitioned and implicit-explicit general linear methods for ordinary
  differential equations.
\newblock {\em J. Sci. Comput.}, 61(1):119--144, 2014.

\end{thebibliography}

\end{document}